\documentclass[a4paper,11pt]{article}

\title{Mazur's isogeny theorem}
\author{Philippe Michaud-Jacobs}
\date{\vspace{-4ex}}
\usepackage{amsmath} 
\usepackage{amssymb} 
\usepackage{mathtools}
\usepackage{amsthm}

\newtheorem{theorem}{Theorem}[section]
\newtheorem*{theorem*}{Theorem}

\newtheorem {lemma}[theorem]{Lemma}
\newtheorem {proposition}[theorem]{Proposition}
\theoremstyle{definition}

\newtheorem{remark}[theorem]{Remark}
\newtheorem*{remark*}{Remark}

\usepackage[numbers]{natbib}
\usepackage{cite}
\usepackage{enumerate}
\usepackage{xcolor}

\providecommand{\Q}{\mathbb{Q}}
\providecommand{\Z}{\mathbb{Z}}

\newcommand{\Addresses}{{
  \bigskip
  \footnotesize

 \textsc{Mathematics Institute, University of Warwick, CV4 7AL, United Kingdom}\par\nopagebreak
  \textit{E-mail address}: \texttt{p.rodgers@warwick.ac.uk}
}}

\let\svthefootnote\thefootnote
\newcommand\freefootnote[1]{%
  \let\thefootnote\relax%
  \footnotetext{#1}%
  \let\thefootnote\svthefootnote%
}

\begin{document}

\maketitle

\begin{abstract}
Mazur's isogeny theorem states that if $p$ is a prime for which there exists an elliptic curve $E / \mathbb{Q}$ that admits a rational isogeny of degree $p$, then $ p \in \{2,3,5,7,11,13,17,19,37,43,67,163 \} $. This result is one of the cornerstones of the theory of elliptic curves and plays a crucial role in the proof of Fermat's Last Theorem. In this expository paper, we overview Mazur's proof of this theorem, in which modular curves and Galois representations feature prominently.
\end{abstract}


\section{Introduction}


The study of elliptic curves plays a fundamental role in number theory. The maps defined on any mathematical object provide insight into their structure. In the case of elliptic curves, we are most interested in \emph{isogenies}. 
\freefootnote{\emph{Date}: \date{\today}.}
\freefootnote{\emph{Keywords}: Elliptic curve, isogeny, Galois representation, modular curve.}
\freefootnote{\emph{MSC2020}: 11F80, 11G05, 11G18, 11-02.}
\freefootnote{The author is supported by an EPSRC studentship and has previously used the name Philippe Michaud-Rodgers.}
The set of points of an elliptic curve forms an abelian group, and an isogeny between elliptic curves is a non-constant morphism that preserves this group structure. An isogeny between elliptic curves defined over $\Q$ is said to be rational if it can be represented by a rational map with coefficients in $\Q$, and its degree is the size of its kernel. Rational isogenies of prime degree form the basic building blocks of isogenies between elliptic curves. Mazur's isogeny theorem, proven in 1978, provides a complete classification of the possible prime degrees of rational isogenies.

\begingroup
\renewcommand\thetheorem{1}
\begin{theorem}[Mazur's isogeny theorem, {\citep[Theorem~1]{ratisog}}]\label{Main} Let $p$ be a prime such that there exists an elliptic curve $E / \Q$ that admits a rational isogeny of degree $p$. Then \[ p \in \{2,3,5,7,11,13,17,19,37,43,67,163 \}.\]
\end{theorem}
\endgroup

Mazur's isogeny theorem is important for several reasons. Apart from significantly furthering our understanding of elliptic curves, its proof introduced many deep and original concepts that still play a crucial role in modern research. Furthermore, Mazur's isogeny theorem provides a key step in the proof of Fermat's Last Theorem (and its many variants and generalisations).

The aim of this expository paper is to overview the proof of Mazur's isogeny theorem. We will, for the most part, follow Mazur's original proof of this result, omitting many of the technical details and focusing on making the proof as accessible as possible. We hope that a reader with some knowledge of elliptic curves, algebraic geometry, and number theory (say to the level of a beginning graduate student) will be able to follow our exposition of the proof. There are several expository articles that have been written covering parts of the material we present here (see \citep{dar, edix, mazurnotes, mazur_serre, reb, serre_b}), and we hope that our modern and simplified exposition will serve as a complement to the existing literature.

\bigskip

\noindent  I would like to thank the Bhaskaracharya Pratishthana Institute for the opportunity to give a talk on Mazur's isogeny theorem, from which this paper stems. I would also like to thank my supervisors, Samir Siksek and Damiano Testa, for helping me understand many details of the proof of Mazur's isogeny theorem. Finally, I would like to thank the anonymous referee for a careful reading of the paper and some valuable suggestions.


\section{Key concepts}


In this section we cover some important concepts related to isogenies of elliptic curves. For further background we recommend \citep{firstcourse} and \citep{silverman}. Let $E_1$ and $E_2$ be elliptic curves over $\Q$. An \textbf{isogeny} $\varphi:E_1 \rightarrow E_2$ is a non-constant morphism of curves that induces a group homomorphism from $E_1(\overline{\Q})$ to $E_2(\overline{\Q})$. We say that the isogeny $\varphi$ is \textbf{rational} (or defined over $\Q$) if it can be represented by a rational map with coefficients in $\Q$. Equivalently, $\varphi^\tau = \varphi$ for any $\tau \in G_\Q$, where we write $G_\Q = \mathrm{Gal}(\overline{\Q}/\Q)$. The \textbf{degree} of an isogeny is defined to be its degree as a morphism of curves, which is equal to the size of its kernel. If $p$ is a prime, we say that an isogeny is a \textbf{$p$-isogeny} if it has degree $p$. We say that $E_1$ \textbf{admits a rational isogeny} if there exists some elliptic curve $E_3 / \Q$ and a rational isogeny $\varphi : E_1 \rightarrow E_3$.

Let $E / \Q$ be an elliptic curve and let $p$ be a prime. Our aim is to study whether or not $E$ admits a rational $p$-isogeny. We first see how to rephrase this in terms of Galois representations.
Given $n \geq 1$, write $E[p^n] \subset E(\overline{\Q})$ for the $p^n$-torsion points of $E$, and write $T_p(E) = \varprojlim_n E[p^n]$ for the $p$-adic Tate module of $E$. Then $T_p(E) \cong \Z_p \times \Z_p$, and by choosing a $\Z_p$-basis for $T_p(E)$ we obtain the $p$-adic Galois representation attached to $E$, which describes the action of the absolute Galois group $G_\Q$ on $T_p(E)$:
\[ \rho_{E,p} : G_\Q \rightarrow \mathrm{GL}_2(\Q_p). \]
Our main object of study will in fact be the mod $p$ Galois representation, which we denote $\overline{\rho}_{E,p}$, which describes the action of $G_\Q$ on $E[p] \cong \Z / p\Z \times \Z / p\Z$: \[ \overline{\rho}_{E,p} : G_\Q \rightarrow \mathrm{GL}_2(\mathbb{F}_p). \]

\begin{lemma}\label{equiv} Let $E / \Q$ be an elliptic curve and let $p$ be a prime. The following are equivalent:
\begin{enumerate}[(i)]
\item $E$ admits a rational $p$-isogeny;
\item $E(\overline{\Q})$ contains a rational (i.e. $G_\Q$-stable) subgroup of order $p$; 
\item $\overline{\rho}_{E,p}$ is reducible.
\end{enumerate}
\end{lemma}

\begin{proof} The kernel of a rational $p$-isogeny is a rational subgroup of order $p$, so (i) implies (ii). Conversely, quotienting by a rational subgroup of order $p$ gives rise to a rational $p$-isogeny.

Next, suppose $\overline{\rho}_{E,p}$ is reducible. By choosing an appropriate basis, say $\{R_1$, $R_2\}$ of $E[p]$, we have \[\overline{\rho}_{E,p} \sim \left( \begin{smallmatrix} \lambda & * \\ 0 & \lambda'  \end{smallmatrix} \right), \] for characters $\lambda, \lambda' : G_\Q \rightarrow \mathbb{F}_p^\times$. Here, $\sim$ denotes an isomorphism of representations. Then for $\tau \in G_\Q$, $R_1^\tau = \lambda(\tau)R_1$, so $\langle R_1 \rangle$ is a rational subgroup, so (ii) implies (iii). Conversely, if $E(\overline{\Q})$ contains a rational subgroup of order $p$, then this is a non-trivial proper $G_\Q$-submodule of $E[p]$, so the representation $\overline{\rho}_{E,p}$ is reducible. \end{proof}

Next, we introduce the notion of a Frobenius element at a prime. Let $q$ be a prime with $q \neq p$. We write $D_q$ for the decomposition group at $q$ and $I_q$ for the inertia group at $q$. We have $I_q \subset D_q \subset G_\Q$, and $D_q / I_q \cong G_{\mathbb{F}_q}$. We define a \textbf{Frobenius element at $q$}, which we write $\sigma_q \in D_q \subset G_\Q$ to be any preimage of the Frobenius endomorphism under the natural quotient map $D_q \rightarrow G_{\mathbb{F}_q}$. We note that a Frobenius element $\sigma_q \in D_q$ is only well-defined up to the inertia group $I_q$ (the kernel of the quotient map $D_q \rightarrow G_{\mathbb{F}_q}$). Although the absolute Galois group $G_\Q$ is a very complex object, these Frobenius elements will give us something concrete to work with.

Another object that will play an important role is the mod $p$ cyclotomic character. Let $\zeta_p \in \overline{\Q}$ denote a primitive $p$th root of unity. For any $\tau \in G_\Q$, $\tau(\zeta_p)$ is also a primitive $p$th root of unity, so  $\zeta_p^\tau = \zeta_p^{a_\tau}$, for some $a_\tau \in \{1, \dots, p-1\}$ . We define the \textbf{mod $p$ cyclotomic character}, $\chi_p$, to be  \begin{align*} \chi_p: G_\Q & \rightarrow \mathbb{F}_p^\times \\ \tau & \mapsto a_\tau, \quad \text{where } \zeta_p^\tau = \zeta_p^{a_\tau}. \end{align*} We now state three important properties of $\chi_p$ that we will use.

\begin{lemma}[{\citep[pp.~386--393]{firstcourse}}]\label{chip} Let $p$ and $q$ be distinct primes and let $\sigma_q$ denote a Frobenius element at $q$. Then \begin{enumerate} [(i)]
\item $\chi_p$ is unramified at $q$ (meaning $\chi_p(I_q) = 1$);
\item $\chi_p(\sigma_q) = q \pmod{p} \in \mathbb{F}_p^\times$;
\item if $E / \Q$ is an elliptic curve, then $\det(\overline{\rho}_{E,p}) = \chi_p$.
\end{enumerate}
\end{lemma}

The final object we introduce is the modular curve $X_0(p)$. This is an algebraic curve defined over $\Q$ whose points parametrise elliptic curves equipped with a $p$-isogeny. The key facts we will need are the following (see \citep[p.~212]{edix} and \citep{firstcourse} for example):
\begin{itemize}
\item $X_0(p)$ is an algebraic curve over $\Q$ and admits a model that has good reduction at every prime $q \neq p$. In the language of schemes, $X_0(p)$ admits a smooth model over $\mathrm{Spec}(\Z[1/p])$. We write $\tilde{x} \in X_0(p)(\mathbb{F}_q)$ for the reduction mod $q$ of a point $x \in X_0(p)(\Q)$.
\item $X_0(p)$ has two distinguished rational points, the cusps, which we denote $\infty,0 \in X_0(p)(\Q)$.
\item If $E / \Q$ is an elliptic curve that admits a rational $p$-isogeny, $\varphi$, then the pair $(E, \varphi)$ gives rise to a rational point $[E, \varphi] = x \in X_0(p)(\Q) \backslash \{ \infty, 0 \}$.
\item $X_0(p)$ comes equipped with the $j$-map, $j:X_0(p) \rightarrow \mathbb{P}^1$, that satisfies $j([E,\varphi]) = j(E)$ and $j(\infty) = j(0) = \infty \in \mathbb{P}^1$ (so the cusps are the poles of the $j$-map). Here, $j(E)$ denotes the $j$-invariant of the elliptic curve.
\item $X_0(p)$ comes equipped with an involution $w_p: X_0(p) \rightarrow X_0(p)$ defined over $\Q$, known as the Atkin--Lehner involution. It satisfies $w_p(\infty) = 0$.
\end{itemize}


\section{Formal immersions and modular curves}


Our aim in this section will be to prove the following result.

\begin{theorem}[Mazur, {\citep[Corollary~4.4]{ratisog}}] \label{mult} Let $E / \Q$ be an elliptic curve and let $p > 19$ be a prime for which $E$ admits a rational $p$-isogeny. Then $E$ has potentially good reduction at any prime $q \notin \{2,p\}$.
\end{theorem}

We recall that an elliptic curve $E/ \Q$ has potentially good reduction at a prime $q$ if $v_q(j(E)) \geq 0$, and potentially multiplicative reduction at a prime $q$ if $v_q(j(E)) < 0$. Theorem \ref{mult} in fact holds if $q = p$, and also for primes $p =11$ or $p > 13$. We will restrict to $p > 19$ since the primes $p \leq 19$ appear in Theorem \ref{Main} anyway.

The proof of Theorem \ref{mult} forms the majority of the proof of Mazur's torsion theorem \citep[Theorem~8]{eisenstein}, which classifies the possible torsion subgroups of elliptic curves defined over the rationals. The expository papers \citep{dar, edix, reb} discuss Mazur's torsion theorem and its generalisations to number fields in some detail. 

We start by seeing how we can deduce Theorem \ref{mult} from a statement about the modular curve $X_0(p)$, using the facts listed at the end of the previous section. Suppose $E$ is an elliptic curve with a rational $p$-isogeny $\varphi$ with $p > 19$, and suppose (for a contradiction) that $q \neq 2,p$ is a prime of potentially multiplicative reduction for $E$. Write $x = [E,\varphi] \in X_0(N)(\Q)$ for the non-cuspidal point that the pair $(E,\varphi)$ gives rise to. Then $v_q(j(x)) = v_q(j(E)) < 0$, and it follows that $\tilde{x} = \tilde{\infty}$ or $\tilde{0}$ in  $X_0(p)(\mathbb{F}_q)$, where we recall that $\sim$ denotes reduction mod $q$. By applying the Atkin--Lehner involution $w_p$ to $x$ if necessary (which swaps the cusps), we may assume that $\tilde{x} = \tilde{\infty}$. Since $x \neq \infty$, we may deduce Theorem \ref{mult} from the following proposition.

\begin{proposition}\label{multprop} Let $p > 19$ be a prime and let $x \in X_0(p)(\Q)$. Let $q \neq 2,p$ be a prime and suppose that $\tilde{x} = \tilde{\infty}$ in $X_0(p)(\mathbb{F}_q)$. Then $x = \infty$.
\end{proposition}

We will prove this proposition by constructing a certain quotient of the Jacobian of $X_0(p)$ and by using the theory of formal immersions.


\subsection{Formal immersions}


In order to discuss formal immersions, we will appeal to the language of schemes. In particular, for a prime $q$, we will work with a morphism $f : X \rightarrow Y$ of schemes over $\Z_q$. We refer to \citep{derthesis} for a more general set-up and further background on formal immersions. For readers less familiar with the language of schemes, it is reasonable to think of $X$ and $Y$ as varieties over $\Q$ together with their mod $q$ reductions. Given a point $x \in X$, we will write $\mathcal{O}_{X,x}$ for the local ring at $x$, and denote by $\mathfrak{m}_x$ its maximal ideal. We write $\hat{\mathcal{O}}_{X,x}$ for its completion, which we may identify with the ring of power series $\Z_q[[u_1, \dots, u_t]]$, where the $u_i$ are local parameters at $x$. We write $\hat{\mathfrak{m}}_x = \mathfrak{m}_x \cdot \hat{\mathcal{O}}_{X,x}$ for the maximal ideal of the completion. Also, we will write $\mathrm{Cot}_x(X) = \mathfrak{m}_x / \mathfrak{m}_x^2$ for the cotangent space of $X$ at $x$, which can be viewed as the space of functions that vanish exactly once at $x$.

The morphism $f : X \rightarrow Y$ induces the following pullback maps: \begin{align*} f^*: \mathcal{O}_{Y,f(x)} & \longrightarrow \mathcal{O}_{X,x}, \\
 f^* : \mathrm{Cot}_{f(x)}(Y) & \longrightarrow \mathrm{Cot}_x(X), \\
\hat{f}^*: \hat{\mathcal{O}}_{Y,f(x)} & \longrightarrow \hat{\mathcal{O}}_{X,x}.
\end{align*} 
We say that $f$ is a \textbf{formal immersion} at $x \in X$ if \[ \hat{f}^*: \hat{\mathcal{O}}_{Y,f(x)} \rightarrow \hat{\mathcal{O}}_{X,x} \] is surjective. 

\begin{lemma}\label{form_inj} Let $f : X \rightarrow Y$ be a morphism of schemes over $\Z_q$ such that $f$ is a formal immersion at $x \in X(\mathbb{F}_q)$. Let $P, Q \in X(\Z_q)$ be such that $\tilde{P} = \tilde{Q} = x \in X(\mathbb{F}_q)$ and suppose $f(P) = f(Q)$. Then $P = Q$.
\end{lemma}

This lemma may be viewed as an analogue of Hensel's lemma. It says that if two points agree as $\mathbb{F}_q$-points, and that some further conditions are satisfied (which we may compare with the derivative criterion of Hensel's lemma), then the points agree as $\Z_q$-points. 

\begin{proof} We may restrict to the case of $X$ and $Y$ being affine schemes. In order to verify that $P = Q$, we will check that the functions on $X$ evaluated at $P$ and $Q$ take the same values. Let $u \in \mathcal{O}_{X,x}$ be any regular function at $x$. The fact that $\tilde{P} = \tilde{Q} = x$ means that $u$ is also regular at $P$ and $Q$, and it will be enough to show that $u(P) = u(Q)$. We view $u \in \hat{\mathcal{O}}_{X,x}$ via the inclusion $\mathcal{O}_{X,x} \hookrightarrow \hat{\mathcal{O}}_{X,x}$. The fact that $f$ is a formal immersion at $x$ means that we can choose $w \in  \hat{\mathcal{O}}_{Y,f(x)}$ such that $\hat{f}^*(w) = u$. Then \[ u(P) = \hat{f}^*(w)(P) = w(f(P)) = w(f(Q)) = \hat{f}^*(w)(Q) = u(Q), \] as required. 
\end{proof}

Lemma \ref{form_inj} is the key property of formal immersions that we will exploit. However, in order to check that a certain map is a formal immersion, we will do so by using the following lemma.

\begin{lemma} \label{form_cot} Let $f : X \rightarrow Y$ be a morphism of schemes over $\Z_q$. Let $x \in X$ and suppose that $x$ and $f(x)$ have the same residue field. Then $f$ is a formal immersion at $x$ if and only if the map $ f^* : \mathrm{Cot}_{f(x)}(Y)  \rightarrow \mathrm{Cot}_x(X)$ is surjective.
\end{lemma}

\begin{proof} If $\hat{f}^*: \hat{\mathcal{O}}_{Y,f(x)}  \rightarrow \hat{\mathcal{O}}_{X,x}$ is surjective, then $\hat{f}^*(\hat{\mathfrak{m}}_{f(x)}) = \hat{\mathfrak{m}}_{x}$ and it follows that $f^*(\mathfrak{m}_{f(x)}) = \mathfrak{m}_x$, so $f^*$ will be surjective on the cotangent spaces. 

For the converse, we first choose elements  $u_1, \dots, u_t \in\mathfrak{m}_{f(x)}$ such that the $u_i \pmod{\mathfrak{m}_{f(x)}^2}$ span the cotangent space $\mathrm{Cot}_{f(x)}(Y)$. Since $f^*$ is surjective on  $\mathrm{Cot}_{f(x)}(Y)$, the elements $f^*(u_i) \pmod{\mathfrak{m}_x^2}$ span $\mathrm{Cot}_x(X)$. Then, by Nakayama's lemma (as stated in \citep[Proposition~2.8]{AM} for example), the elements $f^*(u_i)$ generate $\mathfrak{m}_x$. It follows that $f^*(\mathfrak{m}^n_{f(x)}) = \mathfrak{m}_x^n$ for any $n \geq 1$. We then see that for any $n \geq 1$ we have a surjection \[ f^* : \frac{\mathfrak{m}^{n}_{f(x)}}{\mathfrak{m}^{n+1}_{f(x)}} \longrightarrow  \frac{\mathfrak{m}^{n}_{x}}{\mathfrak{m}^{n+1}_x}. \] Since $x$ and $f(x)$ have the same residue field, the map induced by $f^*$ between the associated graded rings of $\mathcal{O}_{Y, f(x)}$ and $\mathcal{O}_{X,x}$ is surjective, and it follows that the map $\hat{f}^*$ on the completed local rings is surjective by \citep[Lemma~10.23]{AM}.
\end{proof}


\subsection{The Jacobian and the Eisenstein quotient}


We now return to the notation used at the start of this section, namely that $x \in X_0(p)(\Q)$, and $\tilde{x} = \tilde{\infty}$ in $X_0(p)(\mathbb{F}_q)$. We would like to mimic the set up of Lemma \ref{form_inj} with the points $x$ and $\infty$, which we now view as $\Z_q$ points on $X_0(p)$. In particular, we need to construct a map $f : X_0(p) \rightarrow Y$, for some $Y$, which is a formal immersion at $\tilde{x}$ and satisfies $f(x) = f(\infty)$. Focusing on the second condition, a natural first step is to consider the Abel--Jacobi map with base point $\infty$: \begin{align*} \iota: X_0(p)  & \longrightarrow J_0(p) \\ y & \longmapsto [y - \infty]. \end{align*} Here, $J_0(p)$ denotes the Jacobian of $X_0(p)$, and $[y-\infty]$ denotes the divisor class of $y-\infty$. Although $\iota(x) = [x- \infty]$ and $\iota(\infty) = [0]$ need not be equal, we do have that their reductions mod $q$ are equal: \[ \widetilde{\iota(x)} = \widetilde{\iota(\infty)}. \] Now, if we knew that $\iota(x) \in J_0(p)(\Q)_{\mathrm{tors}}$, then by injectivity of reduction on the torsion of an abelian variety \citep[Appendix]{katz}, we would be able to conclude that $\iota(x) = \iota(\infty)$.

Unfortunately, $J_0(p)(\Q)$ need not be finite; that is, its \emph{rank} over $\Q$ may be greater than $0$. Nevertheless, if $A_p$ is an abelian variety that is a quotient of $J_0(p)$, so that we have a map $J_0(p) \rightarrow A_p$, then composing with the Abel--Jacobi map, we obtain a map \[ f_p : X_0(p) \rightarrow A_p. \] If we can choose $A_p$ such that its rank is $0$ over $\Q$, then applying the same argument as above, we could conclude that $f_p(x) = f_p(\infty)$. 

Before seeing whether such a rank $0$ quotient exists, let us see that the map $f_p$ is a formal immersion when $A_p$ is an \textbf{optimal} quotient, with the word \emph{optimal} meaning that both $A_p$ and the kernel of the quotient map from $J_0(p)$ to $A_p$ are abelian varieties (see \citep[p.~140]{ratisog}).

\begin{proposition}[{\citep[Proposition~3.2]{ratisog}}] \label{form_inf} Let $q \neq {2,p}$ be a prime. Let $A_p$ be a non-trivial optimal quotient of $J_0(p)$. Then the map $f_p: X_0(p) \rightarrow A_p$ is a formal immersion at $\tilde{\infty} \in X_0(p)(\mathbb{F}_q)$ (we say that $f_p$ is a formal immersion at $\infty$ in characteristic $q$).
\end{proposition}

\begin{proof}[Proof (sketch)] We have that $f_p(\tilde{\infty}) = \tilde{0} \in A_p(\mathbb{F}_q)$ (where $\tilde{0}$ denotes the identity element of $A_p(\mathbb{F}_q)$ rather than the the reduction of the zero cusp), and so we would like to prove that the map \[f_p^* : \mathrm{Cot}_{\tilde{0}}(A_p) \rightarrow \mathrm{Cot}_{\tilde{\infty}}(X_0(p))\] is surjective.

Since $q >2$, by \citep[Corollary~1.1]{ratisog} we have an injection  $\mathrm{Cot}_{\tilde{0}}(A_p) \hookrightarrow \mathrm{Cot}_{\tilde{0}}(J_0(p))$. Using \citep[p.~214]{edix}, we may identify $\mathrm{Cot}_{\tilde{0}}(J_0(p))$ with the $\Z_q$-module of cusp forms of weight $2$ on $\Gamma_0(p)$, and we may view $\underline{q} = e^{2\pi i z}$ (for $z$ in the upper half-plane) as a uniformiser at $\infty$ and take $d\underline{q}$ as a basis for  $\mathrm{Cot}_{\tilde{\infty}}(X_0(p))$. Using these identifications, the map $f_p^*$ then corresponds to: \begin{align*} f_p^* : \mathrm{Cot}_{\tilde{0}}(A_p) & \longrightarrow \mathrm{Cot}_{\tilde{\infty}}(X_0(p)) \\ \sum_{n=1}^\infty a_n \underline{q}^{n-1} d\underline{q} & \longmapsto a_1. \end{align*} Since the right-hand side is one-dimensional, it suffices to show that this map is non-zero. Since $A_p$ is a non-trivial optimal quotient, using \citep[p.~140]{ratisog} we may choose a cusp form $\sum_{n=1}^\infty a_n \underline{q}^{n-1} d\underline{q}$ that is an eigenvector for the Hecke operators (see (\ref{heckop}) below for how these are defined), and this cusp form will satisfy $a_1 \neq 0$ (for if $a_1 = 0$, by applying the Hecke operators to this form, all the other coefficients would also be $0$).
\end{proof}

Thanks to Proposition \ref{form_inf} and the discussion preceeding it, we will be able to prove Theorem \ref{mult} if we can construct a non-trivial optimal quotient of $J_0(p)$ that has rank $0$ over $\Q$. We will see how Mazur did this by constructing the \textbf{Eisenstein quotient} of $J_0(p)$ in \citep{eisenstein}. We will then simply state its key properties.

We first define, for a prime $\ell \neq {p}$, the Hecke operator $T_\ell$ as an element of the endomorphism ring of the Jacobian $J_0(p)$. We in fact already mentioned these Hecke operators in the proof of Proposition \ref{form_inf}. Write $\underline{q} = e^{2 \pi i z}$ for $z$ in the upper half-plane. On $\underline{q}$-expansions, we have \begin{equation} \label{heckop} T_\ell \left( \sum_{n=1}^\infty a_n \underline{q}^{n-1} d \underline{q} \right) = \sum_{n = 1}^\infty a_{\ell n} \underline{q}^{n-1}d \underline{q} + \sum_{n=1}^\infty a_n \underline{q}^{\ell n-1} d\underline{q}.  \end{equation} This gives rise to a map on the cotangent space (as in the proof of Proposition \ref{form_inf}), which in turn induces a map on $J_0(p)$ (see \citep[p.~102]{ribstein}).
Next, the Atkin--Lehner involution $w_p$ on $X_0(p)$ also induces (by extending to divisors) an element of the endomorphism ring of $J_0(p)$. We now define the Hecke algebra to be the $\Z$-algebra defined by \[ \mathbb{T} := \langle w_p, ~ T_\ell : \ell \neq p \rangle. \]
We then define the \textbf{Eisenstein ideal} as \[ \mathbb{I} := \langle w_p+1, ~ T_\ell - \ell -1 : \ell \neq p \rangle, \] and the \textbf{Eisenstein quotient} of the Jacobian $J_0(p)$ as \[ J_e(p) := \frac{J_0(p)}{\left(\bigcap_{k=1}^{\infty} \mathbb{I}^k \right) J_0(p)}. \] 
The Eisenstein quotient is an optimal quotient of $J_0(p)$ and we will now state the key property of $J_e(p)$ that we need. 

\begin{theorem}[Mazur, {\citep[Theorem~4]{eisenstein}}] The Eisenstein quotient $J_e(p)$ is a non-trivial abelian variety of rank $0$ over $\Q$ .
\end{theorem}

We note that Mazur proves a lot more about $J_e(p)$ in \citep{eisenstein}. We now put the different parts together to prove Proposition \ref{multprop}, which in turn implies Theorem \ref{mult}.

\begin{proof}[Proof of Proposition \ref{multprop}] Let $J_e(p)$ denote the Eisenstein quotient of $J_0(p)$ and write $f_p: X_0(p) \rightarrow J_0(p) \rightarrow J_e(p)$ for the composition of the Abel--Jacobi map with base point $\infty$ and the quotient map to $J_e(p)$. We have that $\widetilde{f_p(x)} = \widetilde{f_p(\infty)}$, and since the rank of $J_e(p)(\Q)$ is $0$, and $q > 2$, we have that $f_p(x) = f_p(\infty)$. By Proposition \ref{form_inf}, the map $f_p$ is a formal immersion at $\tilde{x}$, and we may then apply Lemma \ref{form_inj} to conclude that $x = \infty$.
\end{proof}

\begin{remark} \label{rem} An alternative to using the Eisenstein quotient is to instead use the \emph{winding quotient} of $J_0(p)$, constructed by Merel in \citep{merel}, some years later than Mazur's construction of the Eisenstein quotient. The winding quotient is conjecturally the \emph{largest} quotient of $J_0(p)$ that has has rank $0$ over $\Q$, and the proof of the fact that the winding quotient is non-trivial and has rank $0$ over $\Q$ is more straightforward than the analogous result for the Eisenstein quotient. However, this fact still relies on many deep results. The construction of the winding quotient allowed Merel to prove the uniform boundedness theorem for torsion subgroups of elliptic curves over number fields \citep[Theorem~1]{merel}.
\end{remark}


\section{Galois representations}


In this section we will suppose that $E / \Q$ is an elliptic curve and that $p > 19$ is a prime such that $E$ admits a rational $p$-isogeny, $\varphi$. By Theorem \ref{mult}, we know that any prime $q \neq 2,p$ is a prime of potentially good reduction for $E$. In this section, we will study the mod $p$ Galois representation of $E$ and ultimately prove Mazur's isogeny theorem. We will broadly follow Mazur's original proof, although we will formulate the intermediate results a little differently.


\subsection{The isogeny character}


By Lemma \ref{equiv}, the mod $p$ Galois representation of $E$ is reducible, and \[\overline{\rho}_{E,p} \sim \left( \begin{smallmatrix} \lambda & * \\ 0 & \lambda'  \end{smallmatrix} \right), \] for characters $\lambda, \lambda' : G_\Q \rightarrow \mathbb{F}_p^\times$. Moreover, by Lemma \ref{chip}, we have $\lambda \lambda' = \chi_p$.  The character $\lambda: G_\Q \rightarrow \mathbb{F}_p^\times$ is called the \textbf{isogeny character} of the pair $(E,\varphi)$ and it will be our main object of study. This isogeny character describes how the absolute Galois group acts on the kernel of $\varphi$. Indeed, if $\ker(\varphi) = \langle R \rangle$, then for $\tau \in G_\Q$, we have that \[ R^\tau = \lambda(\tau) \cdot R. \]

The following result is a crucial step towards proving Theorem \ref{Main}. It allows us to directly relate $\lambda$ to the mod $p$ cyclotomic character $\chi_p$.

\begin{theorem}\label{lam_chi}  Let $E / \Q$ be an elliptic curve and let $p > 19$ be a prime for which $E$ admits a rational $p$-isogeny. Write $\lambda$ for the corresponding isogeny character. Then \[ \lambda^{12} = \chi_p^s, ~ \text{ where } s \in \{0,4,6,8,12\}. \] Moreover, if $s = 6$ then $p \equiv 3 \pmod{4}$. 
\end{theorem}

We call the integer $s \in \{0,4,6,8,12\}$ appearing in this theorem the \textbf{isogeny signature} of $\lambda$, or of the pair $(E,\varphi)$.

\begin{proof} The proof of this result is somewhat technical and we break it down into several steps. We mainly follow \citep[2183--2184]{bornes}.

\textbf{Step 1:} We first prove the statement on the inertia group $I_p \subset G_\Q$, namely that \[ \lambda^{12} |_{I_p} = \chi_p^s |_{I_p}, ~ \text{ where }  s \in \{0,4,6,8,12\}, \] and that if $s = 6$ then  $p \equiv 3 \pmod{4}$. We identify $I_p$ as a subgroup of $G_{\Q_p}$ via $I_p \subset D_p \cong G_{\Q_p}$, and view $E$ as an elliptic curve over $\Q_p$.

\textbf{Step 1a:} Suppose $E$ has potentially multiplicative reduction at $p$. Then according to whether $E$ has split or non-split multiplicative reduction, $E$ is either isomorphic to a Tate curve, or the quadratic twist of a Tate curve \citep[Appendix C.~14]{silverman}. It follows that \[ \overline{\rho}_{E,p} \sim \theta \cdot \left( \begin{smallmatrix} \chi_p & * \\ 0 & 1  \end{smallmatrix} \right) \] for a quadratic character $\theta$. So $\lambda^2 |_{I_p} = 1 $ or $\chi_p^2 |_{I_p}$, and the statement follows upon taking sixth powers.

\textbf{Step 1b:} Suppose $E$ has potentially good reduction at $p$. Then there exists an extension $K / \Q_p$ such that $E$ attains good reduction at $p$ and such that the ramification degree $e = e(K/\Q_p) \in \{1,2,4,6\}$. We write $I_p'$ for the inertia subgroup of $G_K$, which we may view as a subgroup of $I_p$. We have the following the facts:
\begin{enumerate}[(i)]
\item $\lambda |_{I_p} = \chi_p^a |_{I_p}$ for some integer $a$.  
\item $\lambda^e |_{I_p'} = \chi_p^r|_{I_p'}$ for some integer $0 \leq r \leq e$. 
\end{enumerate}
Here, (i) follows from the classification of characters on the tame inertia group given in \citep[Proposition~5]{serre}, and (ii) is a deeper result of Raynaud (see \citep[Corollary~3.4.4]{raynaud} or \citep[pp.~277-278]{serre}) which uses the fact that $E$ has good reduction over $K$. Then, on $I_p'$, we have that $\chi_p^{ae} = \chi_p^{r}$, so \begin{equation} \label{aer}ae \equiv r \pmod{p-1}. \end{equation} Also, taking 12th powers in (i) and writing $s = 12r/e$ gives  \[ \lambda^{12} |_{I_p} = \chi_p^{12a} |_{I_p} = \chi_p^{\frac{12a e}{ e}} |_{I_p} = \chi_p^s |_{I_p}. \] For each pair $(e,r)$ we can then compute $s = 12r/e$. Recall that $ e \in \{1,2,4,6\}$ and that $0 \leq r \leq e$. Moreover, if $e$ is even then $r$ must also be even by (\ref{aer}). Running through the list of possible pairs $(e,r)$ proves that $s \in \{0,4,6,8,12 \}$. Finally if $s = 6$, then $(e,r) = (4,2)$ and we deduce that $p \equiv 3 \pmod{4}$ from (\ref{aer}). 

\textbf{Step 2:} We deduce the statement on the full Galois group $G_\Q$. Using similar techniques to the above, one can show that $\lambda^{12}$ is unramified away from $p$ (see \citep[pp.~2185--2187]{bornes}). Also, $\chi_p^s$ is unramified away from $p$ (including at the infinite place, since $s$ is even). It follows that the character $\lambda^{12} / \chi_p^s : G_\Q \rightarrow \mathbb{F}_p^\times$ is everywhere unramified, and therefore trivial on $G_\Q$ (using class field theory). 
\end{proof}


\subsection{A root of two polynomials}


Let $q \neq p$ be a prime, and let $\sigma_q \in D_q \subset G_\Q$ denote a Frobenius element at $q$ (defined in Section 2). We are going to study $\lambda(\sigma_q) \in \mathbb{F}_p^\times$. We will see that $\lambda(\sigma_q)$ is a root of two different polynomials, and use this to extract information about the prime $p$.

First of all, by Theorem \ref{lam_chi}, we have that \[ \lambda^{12}(\sigma_q) = \chi_p^s(\sigma_q) = q^s~(\mathrm{mod}~p), \] where $s$ is the isogeny signature of $\lambda$. It follows that $\lambda(\sigma_q)$ is a root of the mod $p$ reduction of the polynomial \begin{equation}\label{poly1} X^{12} - q^s \in \Z[X]. \end{equation}

Next, $\lambda(\sigma_q)$ is an eigenvalue of $\overline{\rho}_{E,p}(\sigma_q)$ and therefore satisfies the characteristic polynomial of this matrix, so \[ \lambda(\sigma_q)^2 - \mathrm{Tr}(\overline{\rho}_{E,p}(\sigma_q)) \lambda(\sigma_q) + q = 0, \] where we have used the fact that $\mathrm{det}(\overline{\rho}_{E,p}(\sigma_q)) = \chi_p(\sigma_q) = q.$ Now, the trace of $\overline{\rho}_{E,p}(\sigma_q)$ is the reduction mod $p$ of the trace of $\rho_{E,p}(\sigma_q)$ on the $p$-adic Tate module of $E$, so $\lambda(\sigma_q)$ is also a root of the mod $p$ reduction of the polynomial \begin{equation}\label{poly2} X^2 - \mathrm{Tr}(\rho_{E,p}(\sigma_q)) X + q \in \Z[X]. \end{equation}

Combining these results together, the polynomials (\ref{poly1}) and (\ref{poly2}) share a root mod $p$ (namely $\lambda(\sigma_q)$), so \begin{equation*} p \mid \mathrm{Res}\left(X^{2} - \mathrm{Tr}(\rho_{E,p}(\sigma_q))X+q, ~ X^{12}-q^s\right),\end{equation*} where $\mathrm{Res}$ denotes the resultant of the two polynomials. This is a good start, but $\mathrm{Tr}(\rho_{E,p}(\sigma_q))$ is an unknown quantity. The following result will help in this regard.

\begin{lemma}[{\citep[Theorem~3]{serretate}}]\label{HasseBd}
Let $E / \Q$ be an elliptic curve and let $p$ be any prime. Let $q \neq p$ be a prime of potentially good reduction for $E$ and let $\sigma_q \in G_\Q$ denote a Frobenius element at $q$. Then $\mathrm{Tr}(\rho_{E,p}(\sigma_q)) \in \Z$ and satisfies \[ \lvert \mathrm{Tr}(\rho_{E,p}(\sigma_q)) \rvert \leq 2\sqrt{q}.\]
\end{lemma}

In the case that $E$ has good reduction at $q$, this follows from the well-known Hasse--Weil bound. We provide a brief proof of the lemma here, as this is an important step in the argument.

\begin{proof} We view $E / \Q_q$ and $\sigma_q \in G_{\Q_q} \cong D_q$. Since $E$ has potentially good reduction at $q$, there exists an extension $K / \Q_q$ such that $E$ attains good reduction over $K$. The key point here is that we can choose $K$ such that $K / \Q_q$ is totally ramified \citep[p.~498]{serretate}, so that both fields have residue field $\mathbb{F}_q$, and we may view $\sigma_q \in G_K$.  Since $E$ has good reduction over $K$, we have that $T_p(E) \cong T_p(\tilde{E})$, where $\tilde{E}$ denotes the reduction of $E$ mod $q$. Under this isomorphism, the element $\sigma_q$ acts on the reduction of $E$ via the usual Frobenius endomorphism, so $\mathrm{Tr}(\rho_{E,p}(\sigma_q)) \in \Z$ satisfies the required inequality by the Hasse--Weil bound. 
\end{proof}

This lemma provides us with a finite set of possibilities for the trace at $\sigma_q$ and immediately leads to the following result.

\begin{proposition}\label{ResProp}
Let $E/ \Q$ be an elliptic curve and let $p > 19$ be a prime such that $E$ admits a rational $p$-isogeny, $\varphi$. Let $s$ be the isogeny signature of $(E,\varphi)$ and let $q \neq p$ be a prime of potentially good reduction for $E$. Then 
  \[ p \mid R_{q,s} \coloneqq \mathrm{lcm}_{\; \lvert a \rvert \leq 2 \sqrt{q}}\left(\mathrm{Res}(X^2-aX+q, X^{12}-q^s) \right). \] 
\end{proposition}

The integer $R_{q,s}$ is independent of the prime $p$ and only depends on the prime $q$ and the isogeny signature $s$. We use the primes $q = 3$ and $q = 5$ and compute the following values: \begin{align*} R_{3,0} = R_{3,12} & = 8131531262400 = 2^6 \cdot 3^2 \cdot 5^2 \cdot 7^2 \cdot 13^2 \cdot 19 \cdot 37 \cdot 97, \\
R_{5,0} = R_{5,12} & = 17072929032886039622400 \\ & = 2^8 \cdot 3^5 \cdot 5^2 \cdot 7^2 \cdot 13^2 \cdot 17 \cdot 31^2 \cdot 37 \cdot 61 \cdot 157 \cdot 229, \\
R_{3,4} = R_{3,8} & = 9815256000 = 2^6 \cdot 3^8 \cdot 5^3 \cdot 11 \cdot 17. 
\end{align*}

From these computations, we see that if $(E,\varphi)$ is an elliptic curve with a rational $p$-isogeny $\varphi$ with $p > 19$, then by Theorem \ref{mult} and Proposition \ref{ResProp} we see that either $s \in \{0, 12 \}$ and $p = 37$, or that $s = 6$. Since $p=37$ appears in our original list of primes in Theorem \ref{Main}, we are reduced to considering the case $s = 6$. 

\subsection{Completing the proof}


In the case $s = 6$, we are unable to obtain any information on the prime $p$, and this is due to the fact that $R_{q,6} = 0$ for all primes $q$. This occurs because the polynomials $X^2 - q$ and $X^{12} - q^6$ share a root and their resultant is therefore $0$. We need to refine our argument in this case. Recall that $p \equiv 3 \pmod{4}$ when $s = 6$. Our aim is to prove the following proposition, and this will complete the proof of Theorem \ref{Main}.

\begin{proposition}\label{class_no} Let $E/ \Q$ be an elliptic curve and let $p > 19$ be a prime such that $E$ admits a rational $p$-isogeny, $\varphi$. Suppose the isogeny signature of $(E,\varphi)$ is $6$. Then $\Q(\sqrt{-p})$ has class number $1$, so $p \in \{43, 67, 163 \}$. 
\end{proposition}

In order to prove this result, we will start by looking in more detail at the isogeny character $\lambda$. 

\begin{lemma}\label{psi} Suppose the isogeny signature of $\lambda$ is $6$. Then \[ \lambda = \psi \cdot \chi_p^{\frac{p+1}{4}}, \] for some character $\psi : G_\Q \rightarrow \mathbb{F}_p^\times$ satisfying $\psi^6 = 1$.
\end{lemma}

\begin{proof} We verify that the character $\psi := \lambda \chi_p^{-\frac{p+1}{4}}$ satisfies $\psi^6 = 1$. Since $\psi^{p-1} = 1$ and $p \equiv 3 \pmod{4}$, it will in fact be enough to check that $\psi^{12} = 1$  (because $\gcd(12, p-1) \mid 6$). We have \[ \psi^{12} = \frac{\lambda^{12}}{\chi_p^{3(p+1)}} = \frac{\lambda^{12}}{\chi_p^{3(p-1)}\chi_p^{6}} = \frac{\lambda^{12}}{\chi_p^6} = 1, \] as required.
\end{proof}

Lemma \ref{psi} relates the character $\lambda$ to $\chi_p$ (rather than relating a power of $\lambda$ to $\chi_p$ as we had previously) and this will allow us to directly consider $\lambda(\sigma_q)$. For small values of $q$ (relative to $p$), this will allow us to obtain a congruence condition on $q$, which we may translate as a splitting condition in the quadratic field $\Q(\sqrt{-p})$ to obtain the following lemma.

\begin{lemma}\label{inert} Suppose the isogeny signature of $(E,\varphi)$ is $6$. If $2 < q < p/4$ then $q$ is inert in $\Q(\sqrt{-p})$.
\end{lemma} 

\begin{proof} Assume for a contradiction that $2 < q < p/4$, but that $q$ is not inert in  $\Q(\sqrt{-p})$. We may rephrase this as a congruence condition and deduce that \[ q^{\frac{p+1}{2}} \equiv q \pmod{p}.\]
Using this congruence and Lemma \ref{psi} we see that  \[ \lambda^2(\sigma_q) = \psi^2(\sigma_q) \chi_p^{\frac{p+1}{2}}(\sigma_q) = \psi^2(\sigma_q) q^{\frac{p+1}{2}} = \psi^2(\sigma_q)q  \] and \[ (\chi_p \lambda^{-1})^2(\sigma_q)  = \frac{q^2}{\psi^2(\sigma_q)q} = \psi^{-2}(\sigma_q) q.\] 
Summing these two identities, we obtain \[ \lambda^2(\sigma_q) + (\chi_p \lambda^{-1})^2(\sigma_q) = q\left(\psi^2(\sigma_q) + \psi^{-2}(\sigma_q) \right). \] Rewriting the left-hand side as $\left(\lambda(\sigma_q) + (\chi_p \lambda^{-1})(\sigma_q)\right)^2 - 2\chi_p(\sigma_q)$    
and recalling that $\mathrm{Tr(\overline{\rho}_{E,p}(\sigma_q))} = \widetilde{\mathrm{Tr}}(\rho_{E,p}(\sigma_q))$, where the $\sim$ denotes reduction mod $p$, we have that 
\[ \widetilde{\mathrm{Tr}}(\rho_{E,p}(\sigma_q))^2 - 2q = q\left(\psi^2(\sigma_q) + \psi^{-2}(\sigma_q) \right). \] Since $\psi^6 = 1$, we have that $\psi^2(\sigma_q)$ is a third root of unity in $\mathbb{F}_p^\times$, and so the right-hand side of this expression if either $2q$ or $-q$ according to whether $\psi^2(\sigma_q)$ is trivial or not. We deduce that \[ p \mid \mathrm{Tr}(\rho_{E,p}(\sigma_q))^2 - rq, \quad \text{ where }r = 1 \text{ or } 4. \] In either case, Lemma \ref{HasseBd} combined with our assumption that $4q < p$ leads to a contradiction.
\end{proof}

Having proven this claim, we can now use some (relatively) straightforward algebraic number theory to prove that the class number of $\Q(\sqrt{-p})$ is $1$. We prove Proposition \ref{class_no} and therefore complete the proof of Theorem~\ref{Main}.

\begin{proof}[Proof of Proposition \ref{class_no}]
Let $\mathfrak{q} \mid q$ be a prime ideal of the ring of integers of $\Q(\sqrt{-p})$ satisfying $2 < \mathrm{Norm}(\mathfrak{q}) < p/4$ with $\mathrm{Norm}(\mathfrak{q})$ odd. Then $2 < q < p/4$, and by Lemma \ref{inert} the ideal $\mathfrak{q}$ is principal. It follows that any ideal $\mathfrak{b}$ satisfying $2 < \mathrm{Norm}(\mathfrak{b}) < p/4$  with $\mathrm{Norm}(\mathfrak{b})$ odd is principal.

Next, we claim that the prime ideals above $2$ are also principal. If $2$ is inert in $\Q(\sqrt{-p})$ then the claim holds, so we may assume $2$ splits in $\Q(\sqrt{-p})$. Then $p \equiv -1 \text{ or } 7 \pmod{16}$. Suppose $p \equiv -1 \pmod{16}$, write $p = -1 + 16t$, and consider \[ \alpha = \frac{3 + \sqrt{-p}}{2} \in \mathcal{O}_{\Q(\sqrt{-p})}.\] Then $ \mathrm{Norm}(\alpha) =  2(1+2t)$, so $\alpha \cdot \mathcal{O}_{\Q(\sqrt{-p})} = \mathfrak{c} \cdot \mathfrak{b},$ for a prime ideal $\mathfrak{c} $ above $2$ (of norm $2$), and an ideal $\mathfrak{b}$ satisfying $ \mathrm{Norm}({\mathfrak{b}}) < \frac{p}{4}$ with $\mathrm{Norm}(\mathfrak{b})$ odd.  The ideal $\mathfrak{b}$ is therefore principal, so $\mathfrak{c}$ is as well. In the case that $ p \equiv 7 \pmod{16}$, we apply the same argument with $\alpha =(1 + \sqrt{-p})/2$.

The Minkowski bound for $\Q(\sqrt{-p})$ is  \[ M_{\Q(\sqrt{-p})} = \frac{2 \sqrt{p}}{\pi} < p/4, \] with the inequality holding since $p > 19 $. So every ideal whose norm is less than the Minkowski bound is principal and we conclude that $\Q(\sqrt{-p})$ has class number $1$.
\end{proof}

This completes the proof of Mazur's isogeny theorem. We take the opportunity to note that for each prime $p$ appearing in Theorem \ref{Main}, there does indeed exist an elliptic curve defined over $\Q$ that admits a rational isogeny of degree $p$.


\section{Further results}


Mazur's isogeny theorem classifies the possible prime degrees of rational isogenies. We start by stating a natural extension of this result, which is to consider rational cyclic isogenies. We say that an isogeny is \textbf{cyclic} if its kernel is a cyclic group.

\begin{theorem}[Mazur, Kenku, {\citep{kenku}}] Let $E/ \Q$ be an elliptic curve. Let $N$ be an integer such that $E$ admits a rational isogeny which is cyclic of degree $N$. Then \[ N \in \{1, 2,  \dots, 19\} \cup \{ 21, 25, 27, 37, 43, 67, 163\}. \]
\end{theorem}

Mazur's isogeny theorem reduces the proof of this theorem to computing the rational points on a finite number of curves $X_0(N)$. This was done in a series of papers by Kenku, culminating in the paper \citep{kenku}. 

A second extension of Mazur's isogeny theorem is to consider isogenies of prime degree over number fields other than $\Q$ (one may naturally extend the definitions presented in Section 2 to number fields). At this time, it has not been possible to obtain such a classification for a single other number field, although some progress has been made towards this, initiated by Momose in \citep{Momose}. It is possible to suitably generalise the results in Section 3 (and this leads to results such as Merel's uniform boundedness theorem mentioned in Remark \ref{rem}), and it is possible to generalise many results presented in Section 4. However, the sticking point is dealing with the case analogous to that of isogeny signature $6$, which we recall required special consideration in Section 4. One possibility is to assume the generalised Riemann hypothesis, in which case it is possible to obtain similar classifications for other number fields not containing an imaginary quadratic field of class number one \citep{B-D}. Another possibility is to assume the elliptic curve is semistable at the primes one is working with (a natural assumption in the framework of variants of Fermat's Last Theorem), as done by the author in \citep{MJ}. 

\bibliographystyle{plainnat}

\Addresses

\end{document}